\documentclass[12pt]{amsart}
\usepackage{amsmath, amsthm, amscd, amsfonts, amssymb, graphicx, color}

\theoremstyle{definition}

\theoremstyle{remark}

\numberwithin{equation}{section}
\theoremstyle{definition}
\newtheorem{thm}{Theorem}[section]

\newtheorem{lem}[thm]{Lemma}

\theoremstyle{definition}
\newtheorem{defn}[thm]{Definition}
\theoremstyle{remark}
\newtheorem{rem}[thm]{\bf Remark}
\theoremstyle{remark}
\numberwithin{equation}{section}
\newtheorem{exam}[thm]{\bf Example}

\newcommand{\A}{\mathcal{A}}

\newcommand{\f}{\mathcal{F}}
\newcommand{\K}{\mathcal{K}}

\newcommand{\D}{\mathcal{D}}
\newcommand{\LL}{\mathcal{L}}

\newcommand{\C}{\mathcal{C}}
\newcommand{\B}{\mathcal{B}}

\newcommand{\p}{\widetilde{p}}
\newcommand{\q}{\widetilde{q}}
\newcommand{\rr}{\widetilde{r}}
\newcommand{\LM}{Lmc(S)}

\begin{document}

\title[Filters and the weakly almost periodic compactification]{Filters and the weakly almost periodic compactification of a semitopological semigroup}


\author[M. A. Tootkaboni]{M. Akbari Tootkaboni\\
Department of Mathematics, Faculty of Basic Science,\\
 Shahed University of Tehran,\\
  Tehran-Iran.}
\address{}
\curraddr{}
\email{tootkaboni.akbari@gmail.com\\
akbari@shahed.ac.ir}
\thanks{}

\date{}

\dedicatory{}

\begin{abstract}
Let $S$ be a semitopological semigroup. The $wap-$ compactification of semigroup S, is a compact
semitopological semigroup with certain universal properties relative to the original semigroup, and
the $Lmc-$ compactification of semigroup $S$ is a universal semigroup compactification of $S$,
which are denoted by $S^{wap}$ and $S^{Lmc}$ respectively. In this paper, an internal construction of the
$wap-$compactification of a semitopological semigroup is constructed as a space of $z-$filters. Also we
obtain the cardinality of $S^{wap}$ and show that if $S^{wap}$ is the one point compactification then
$(S^{Lmc}-S)*S^{Lmc}$ is dense in $S^{Lmc}-S$.
\end{abstract}

\maketitle
\section{\bf Introduction}

A semigroup $S$ which is
also a  Hausdorff topological space is called a semitopological semigroup, if  for each $s\in S$, $\lambda_{s} :S\rightarrow S$ and $r_{s} :S\rightarrow S$ are
continuous, where for each $x\in S$, $\lambda_{s}(x)=sx$ and
$r_{s}(x)=xs$. Notice that if  only  $r_s$, for each $s\in S$, is continuous, $ S$ is called a right topological semigroup.
Throughout this paper $S$  is a \mbox{semitopological} semigroup. Generally, the Stone-$\check{C}$ech compactification $\beta S$ of
space $S$ is defined in two approaches:\\
i)	as the spectrum of $\C\B(S)$, the $C^*-$algebra of
bounded complex valued continuous functions on $S$,(see \cite{Analyson}), or\\
ii)	the collection of all $z$-ultrafilters on $S$, (see \cite{Gil}).

Generally $\beta S$ is not a semigroup. A necessary and sufficient condition for $\beta S$
 to be a semigroup naturally is that $\C\B(S)$ should be an $m-$admissible algebra.
If $S$ is realized as a discrete semigroup, then $\C\B(S)$ will be an $m-$admissible
 algebra and as a result, $\beta S$ will be a semigroup.
 This semigroup, as the collection of all ultrafilters on $S$, has a known operation
 attributed to Glazer. Capability and competence of ultrafilter approach are illustrated
 clearly in \cite{Ultra}, \cite{Gil}, \cite{hindbook} and \cite{zelbook}.
 It is well known that ultrafilters play a prominent role in the Stone-$\check{C}$ech
 compactification of discrete semigroups, see \cite{hindfilter}, \cite{Gil} and \cite{hindbook}. It is natural
 to study the semigroup compactification of a semitopological semigroup as a
collection of $z$-filters, see \cite{Akbari} for more details. This approach sheds a new light on studying the this
kind of compactifications. By getting help from what has been
done  in \cite{Akbari}, has been tried to introduce some new subjects about semigroup compactification
through $z$-filters, that was not applicable without this objects until now. See \cite{Akbari4},\cite{Akbari3} and \cite{Akbari1}.

This paper by getting ideas from what filters do in the Stone-$\check{C}$ech compactification, has been prepared.
Indeed, ideas are taken from \cite{hindfilter} and it has been used in \cite{Akbari2}. \\
In Section 2, semigroup compactification has been introduced
briefly and $m$-admissible subalgebras of $Lmc(S)$ and $wap(S)$ have been explained;
then semigroup compactification is rebuilt as a collection of $z$-filters,
and some Theorems and Definitions of \cite{Akbari2} are presented.

In Section 3, regarding to \cite{Akbari2}, the $wap$-compactification as a quotient space of the $Lmc$- compactification is described, and
we obtain the cardinal of the $wap$-compactification under some conditions and present some examples in this area.

\section{\bf Preliminary}

 A semigroup compactification of   $S$ is a pair
$(\psi,X)$, where $X$ is a compact, Hausdorff, right topological
semigroup and $\psi:S\rightarrow X$ is  continuous homomorphism with dense image
such that for all $s\in S$, the mapping
$x\mapsto\psi(s)x:X\rightarrow X$ is continuous. The last property say that $ \psi[S] $ is  in the topological
center of $X$ \cite[Definition 3.1.1]{Analyson}. Let $\mathcal{F}$ be a $C^*$-subalgebra
 of $\mathcal{CB}(S)$ containing the constant functions. Then the set of all multiplicative means of
 $\f$, called the spectrum of $\f$,  denoted by $S^{\f}$, equipped with the
  Gelfand topology, is a compact Hausdorff topological space. Given $\f$ is left translation invariant if $L_sf=f\circ \lambda_{s}\in\f$ for all
$s\in S$ and $f\in\f$. Then $\f$  is called
$m$-admissible too if the function $s\mapsto(T_\mu f(s))=\mu(L_s f)$ is
in $\f$ for all $f\in\f$ and $\mu\in S^{\f}$. If so, $S^{\f}$ under
the multiplication $\mu\nu=\mu\circ T_\nu$  for $\mu,\nu\in
S^{\mathcal{F}}$, equipped with the Gelfand topology, makes
  semigroup compactification $(\varepsilon, S^{\f})$ of $S$, called the
$\f$-compactification, where $\varepsilon :S\rightarrow
S^{\mathcal{F}}$ is the evaluation mapping. Also
$\varepsilon^{*}:\mathcal{C}(S^{\mathcal{F}})\rightarrow\mathcal{F}$
is an isometric isomorphism and
$\widehat{f}=(\varepsilon^{\ast})^{-1}(f)\in
\mathcal{C}(S^{\mathcal{F}})$ for $f\in\f$ is given by
$\widehat{f}(\mu)=\mu(f)$ for all $\mu\in S^{\f}$. For more details
see \cite[Section 2]{Analyson}.\\
 A function $f\in \mathcal{CB}(S)$ is left multiplicative continuous
if and only if $\mathbf{T}_{\mu}f\in \mathcal{CB}(S)$ for all
$\mu\in\beta S=S^{\mathcal{CB}(S)}$, then
\begin{equation*}
Lmc(S)=\bigcap\{\mathbf{T}^{-1}_{\mu}(\mathcal{CB}(S)):\mu\in\beta S\},
\end{equation*}
and $Lmc(S)$ is the largest $m$-admissible subalgebra of
$\mathcal{CB}(S)$. Then $(S^{Lmc}, \varepsilon)$ is the
universal compactification of $S$, see \cite[Definition 4.5.1 and Theorem
4.5.2]{Analyson}.

A function $f\in \mathcal{CB}(S)$ is said to be weakly almost
periodic if $R_{S}f=\{f\circ r_{s}:s\in S\}$ is relatively compact
(i.e., $\sigma (\mathcal{CB}(S), \mathcal{CB}(S)^{*})$) in
$\mathcal{CB}(S)$. The set of all weakly almost periodic functions
on $S$ is denoted by $wap(S)$,( see 4.2.1 in \cite{Analyson}.)

\begin{thm}
Let $S$ be a semitopological semigroup and $f\in \mathcal{CB}(S)$.
The following statements are
equivalent:\\
$(i)$ $f\in wap(S)$.\\
$(ii)$ $lim_{m\rightarrow\infty}lim_{n\rightarrow\infty}f(s_{m}t_{n})=
lim_{n\rightarrow\infty}lim_{m\rightarrow\infty}f(s_{m}t_{n})$
whenever $\{s_{m}\}$ and $\{t_{n}\}$ are sequences in $S$ such that
all the limits exist.
\end{thm}
\begin{proof}
 See Theorem 4.2.3 in \cite{Analyson}.
\end{proof}
\begin{thm}
Let $S$ be a semitoplogical semigroup. There is a compact
semigroup $S^{wap}$ and a continuous homomorphism
$\varepsilon:S\rightarrow S^{wap}$ such that\\
i) $S^{wap}$ is a semitoplogical semigroup,\\
ii) $\varepsilon(S)$ is dense in $S^{wap}$, and\\
iii) the pair $(S^{wap},\varepsilon)$ is maximal with respect to
these properties in the sense that $\phi$ is a continuous
homomorphism from $S$ to a compact semigroup $T$, and $(T,\phi)$
satisfies (i) and (ii) with $\phi$ replacing $\varepsilon$ and $T$
replacing $S^{wap}$, then there is a continuous homomorphism
$\eta$ from $S^{wap}$ onto $T$ such that $\eta \circ
\varepsilon=\phi$. Moreover, a function $f\in \mathcal{CB}(S)$ extends to $S^{wap}$ if and
only if $f$ is weak almost periodic.
\end{thm}

\begin{proof}
 See Theorem 2.5 in \cite{hindfilter}.
\end{proof}
\begin{lem}
$(i)$ Let $S$ be a semitopological semigroup, and let $f,g\in
wap(S)$ be such that $Range(f)\subseteq \mathbb{R}$ and
$Range(g)\subseteq \mathbb{R}$. Define $h$ by
$h(s)=max\{f(s),g(s)\}$ for each $s\in S$. Then $h\in wap(S)$.\\
$(ii)$ If $S$ is a compact semitopological semigroup, then
$wap(S)=\mathcal{CB}(S)$.
\end{lem}

\begin{proof}
 See Lemma 2.8 in \cite{hindfilter}.
\end{proof}
Now we quote some prerequisite material from \cite{Akbari} for the
description of $(S^{Lmc},\varepsilon)$ in terms of $z-$filters and some Definitions and Theorems
of \cite{Akbari2} is explained that we need. For
$f\in Lmc(S)$, $Z(f)=f^{-1}(\{0\})$ is called zero set for all $f\in Lmc(S)$
and we denote the collection of all zero sets by $Z(Lmc(S))$. For an
extensive account of ultrafilters, the readers may refer to
\cite{Ultra},
 \cite{Gil} and \cite{hindbook}.

\begin{defn}
$\mathcal{A}\subseteq Z(Lmc(S))$ is called a $z-$filter in $Lmc(S)$
if\\
(i) $\emptyset\notin \mathcal{A}$ and $S\in\A,$\\
(ii) if $A,B\in \mathcal{A}$, then $A\bigcap B\in\mathcal{A},$ and\\
(iii) if $A\in \mathcal{A}$, $B\in Z(Lmc(S))$ and $A\subseteq
B$ then $B\in \mathcal{A}$.
\end{defn}
A $z-$filter is a $z-$ultrafilter if it is not contained
properly in any other $z-$filters. The collection of all ultrafilters in $Lmc(S)$ is denoted
by $zu(S)$. If $p,q\in zu(S)$, then the following statements hold.\\
1) If $B\in Z(Lmc(S))$ and for all $A\in p$, $A\cap B\neq\emptyset$ then
$B\in p$,\\
2) if $A, B\in Z(Lmc(S))$ such that $A\cup B\in p$ then $A\in p$ or
$B\in p$, and\\
3) let $p$ and $q$ be distinct $z-$ultrafilters, then there exist $A\in p$ and $B\in q$ such that
$A\cap B=\emptyset$.( See Lemma 2.3 in \cite{Akbari}).

For $x\in S$, we define $\widehat{x}=\{Z(f):f\in Lmc(S),\,\,f(x)=0\}$.
Let $Z(h)\in Z(Lmc(S))$ and $Z(h)\notin \widehat{x}$. Then $M=\{\widehat{f}\in \mathcal{C}(S^{Lmc}):f(x)=0\}$
is a maximal ideal in $\mathcal{C}(S^{Lmc})$
and so $\widehat{h}\notin M$. Thus there exists a $\widehat{f}\in M$ such that $Z(\widehat{f})\cap Z(\widehat{h})=\emptyset$.
Thus $\widehat{x}\cup\{Z(h)\}$ has not the finite intersection property, and so $\widehat{x}$ is a $z-$ultrafilter.

The space $zu(S)$ is equipped with a topology whose base is $\{
(\widehat{A})^{c}: A\in Z(Lmc(S))\}$, where $\widehat{A}=\{p\in zu(S):
A\in p\}$, is a compact space which is not Hausdorff in general. By Lemma 2.7. in \cite{Akbari}, for each $p\in zu(S)$
there exists $\mu\in S^{Lmc}$ such that $\bigcap_{A\in p}\overline{\varepsilon(A)}=\{\mu\}$, and also for each
$\mu\in S^{Lmc}$ there exists $p\in zu(S)$ such that $\bigcap_{A\in p}\overline{\varepsilon(A)}=\{\mu\}$. We say $p\in zu(S)$ converges
to $\mu\in S^{Lmc}$ if $\mu\in\bigcap_{A\in p}\overline{\varepsilon(A)}$. \\
 Now, we define the relation $\sim$ on $zu(S)$ such that $p\sim q$ if and only if
\begin{equation*}
 \bigcap_{A\in p}\overline{\varepsilon(A)}=\bigcap_{B\in
q}\overline{\varepsilon(B)}.
\end{equation*}
 It is obvious that $\sim$ is an
equivalence relation on $zu(S)$ and denote $[[p]]$ as the equivalence class
of $p\in zu(S)$.  So for each $p\in zu(S)$ there is a unique $\mu\in S^{Lmc}$ such that
\[
[[p]]=\{p\in zu(S):\bigcap_{A\in p}\overline{\varepsilon(A)}=\{\mu\}\},
\]
in fact, $[[p]]$ is the collection of all $z-$ultrafilters in $Lmc(S)$ that converge to $\mu\in S^{Lmc}$. Let $\frac{zu( S)}{\sim}$ be the corresponding
quotient space with the quotient map $\pi:zu(S)\rightarrow \frac{zu(S)}{\sim}$.
For every $p\in zu(S)$, define $\widetilde{p}$ by $\widetilde{p}=\bigcap [[p]]$, put
$\widetilde{A}=\{\widetilde{p}:A\in p\}$ for $A\in Z(Lmc(S))$ and
$\mathcal{R}=\{\widetilde{p}:p\in zu(S)\}$. It is obvious that
 $\{(\widetilde{A})^{c}:A\in Z(Lmc(S))\}$ is a basis for a
topology on $ \mathcal{R}$, $\mathcal{R}$ is a Hausdorff and compact
space and also $S^{Lmc}$ and $\mathcal{R}$  are homeomorphic
(see \cite{Akbari}). So we have
$\mathcal{R}=\{\mathcal{A}^{\mu}:\mu\in S^{Lmc}\}$, where
$\mathcal{A}^{\mu}=\widetilde{p}$ and $\A^{\varepsilon(x)}=\widehat{x}$ for $x\in S$.
For all $x,y\in S$, we define
\[
\mathcal{A}^{\varepsilon(x)}*\mathcal{A}^{\epsilon(y)}=\{ Z(f)\in Z(Lmc(S)):\;
Z(\mathbf{T}_{\varepsilon(y)}f)\in \mathcal{A}^{\varepsilon(x)}\}.
\]
It has been shown that $\mathcal{A}^{\varepsilon(x)}*\mathcal{A}^{\varepsilon(y)}=\mathcal{A}^{\varepsilon(xy)}$,
for each $x,y\in S$, see \cite{Akbari}. Now, let $\{x_{\alpha}\}$ and  $\{y_{\beta}\}$ be two nets in $S$, such
that $lim_{\alpha}\varepsilon(x_{\alpha})=\mu$ and
$lim_{\beta}\varepsilon(y_{\beta})=\nu$, for $\mu,\nu\in
S^{Lmc}$. We define
\[
\mathcal{A}^{\mu}*\mathcal{A}^{\nu}=lim_{\alpha}(lim_{\beta}(
\mathcal{A}^{\varepsilon(x_{\alpha})}*\mathcal{A}^{\varepsilon(y_{\beta})})).
\]
This Definition is well-defined and $(\mathcal{R},e)$ is
 a compact right topological semigroup,
where $e:S\rightarrow \mathcal{R}$ is defined by
$e[x]=\widehat{x}$. Also the mapping
$\varphi:S^{Lmc}\rightarrow \mathcal{R}$ defined by
$\varphi(\mu)=\widetilde{p}$, where $\bigcap_{A\in
p}\overline{\varepsilon(A)}=\{\mu\}$, is an isomorphism (see
\cite{Akbari}).

The operation ``$\cdot$" on S extends uniquely to $(\mathcal{R},*)$. Thus
$(\mathcal{R}, e)$ is a semigroup compactification of $(S,\cdot)$, that
$e:S\rightarrow \mathcal{R}$ is an evaluation map. Also $e[S]$ is a subset of the
topological center of $\mathcal{R}$ and
$cl_{\mathcal{R}}(e[S])=\mathcal{R}$. Hence $S^{Lmc}$ and $\mathcal{R}$ are topologically isomorphic and
so $S^{Lmc}\simeq\mathcal{R}$. For more details see
\cite{Akbari}.

A $z-$filter $\mathcal{A}$ on $\LM$ is called a pure $z-$filter if
 for some $z$-ultrafilter $p$, $\mathcal{A}\subseteq p$
implies that $\mathcal{A}\subseteq \widetilde{p}$. For
a $z-$filter $\mathcal{A}\subseteq Z(\LM)$, we define\\
$(i)$ $\overline{\mathcal{A}}=\{\widetilde{p}\in
S^{Lmc}:\mbox{There exists a }z-\mbox{ultrafilter }p$ such that
$\mathcal{A}\subseteq p\},$\\
$(ii)$ $\A^{\circ}=\{A\in \A:
\overline{\A}\subseteq(\overline{\varepsilon(A)})^{\circ}\}.$
\begin{defn}
Let $\A$ and $\B$ be two $z-$filters on $\LM$ and $A\in Z(\LM)$. We
say $A\in\A+\B$ if and only if for each $F\in Z(\LM)$,
$\Omega_{\B}(A)\subset F$ implies $F\in\A$, where
$\Omega_{\B}(A)=\{x\in S:\lambda_{x}^{-1}(A)\in\A\}$.
\end{defn}

Let $\A$ and $\mathcal{B}$  be $z-$filters in $Lmc(S)$
then $\A + \B$ is a $z-$filter, and we define  $\A\odot
\B=\bigcap\overline{\A +\mathcal{B}}$. So $\A\odot \B$ is a pure $z-$filter generated by
$\overline{\A+\B}$.

In this paper, $\mathbb{Q}$ denote rational numbers with natural
topology. Also we replace $\overline{\varepsilon(A)}$ with
$\overline{A}$ for simplicity.\\
Now we describe some Definitions and Theorems of \cite{Akbari2} that applied in next sections.

\begin{defn}
Let $\Gamma$ be a subset of pure $z-$filters in $Lmc(S)$.
 We say $f:\Gamma\rightarrow\bigcup \Gamma$ is a topological choice function
 if for each $\mathcal{L} \in
\Gamma,$ $\overline{\mathcal{L}}\subseteq
(\overline{f(\mathcal{L}}))^{\circ}$.
\end{defn}

\begin{thm}
Let $\Gamma$ be a set of pure $z-$filters in $Lmc(S)$.
Statements $(a)$ and $(a')$ are equivalent, statements $(b)$ and
$(b')$ are equivalent and statement $(c)$ is equivalent to the
conjunction of statements $(a)$ and $(b)$.\\
$(a)$ Given any topological
 choice function $f$ for $\Gamma$, there is a finite subfamily
$\mathcal{F}$ of $\Gamma$ such that
$\overline{S}=S^{Lmc}=\bigcup_{\mathcal{L}\in
\mathcal{F}}(\overline{f(\mathcal{L})})^{\circ}$.\\
 $(a')$  For each $\widetilde{p} \in
S^{Lmc}$ there is some $\mathcal{L} \in \Gamma$ such
that $\mathcal{L}\subseteq \widetilde{p}$.\\
$(b)$ Given distinct $\mathcal{L}$ and $\mathcal{K}$ in $\Gamma$,
there exists $B\in\K^{\circ}$ such that for every $A \in
Z(Lmc(S))$ if
$\overline{S}-(\overline{B})^{\circ}\subseteq
(\overline{A})^{\circ}$ then $A \in \mathcal{L}$.\\
 $(b')$ For each $\widetilde{p} \in
S^{Lmc}$ there is at most one $\mathcal{L} \in \Gamma$
such that $\mathcal{L} \subseteq \widetilde{p}$.\\
 $(c)$ There is an
equivalence relation $R$ on $S^{Lmc}$ such that each
equivalence class is closed in $S^{Lmc}$ and
$\Gamma=\{\,\,\bigcap [\widetilde{p}]_{R}: \widetilde{p} \in
S^{Lmc}\}$.
\end{thm}
\begin{proof}
See Theorem 3.2 in \cite{Akbari2}.
\end{proof}

Let $\Gamma\subseteq Z(Lmc(S))$ be a set of pure
$z-$filters. We define
\begin{equation*}
 A^{*}=\{ \LL\in\Gamma: \overline{\LL}\bigcap \overline{A}\neq
\emptyset\},
\end{equation*}
for each $A\in Z(Lmc(S)$. Then $\{\, (A^{*})^{c}: A\in Z(Lmc(S))\}$
is a base for the topology $\tau$
on $\Gamma$ and $\tau$ is called quotient topology on $\Gamma$ generated
by $\{\, (A^{*})^{c}: A\in Z(Lmc(S))\}$. Let $R$
be an equivalence relation on $S^{Lmc}$ and let
$\Gamma=\{ \cap [\widetilde{p}]_{R}: \widetilde{p}\in
S^{Lmc}\}$. Then for every $A\in Z(Lmc(S))$, we
have $A^{*}=\{\cap[\widetilde{p}]_{R}:\widetilde{p}\in
\overline{A}\}$.  It has been shown that, if $S^{Lmc}/R$ be a Hausdorff space then
with the quotient topology on $\Gamma$, $\Gamma$ and $S^{Lmc}/R$ are homeomorphic (i.e.  $\Gamma\approx S^{Lmc}/R$).( See Theorem 3.5 in \cite{Akbari2}.)
\begin{thm}
Let $\Gamma$ be a set of pure $z-$filters in $Lmc(S)$ with
the quotient topology. There is an equivalence relation $R$ on
$S^{Lmc}$ such that $S^{Lmc}/R$ is Hausdorff,
$\Gamma=\{\, \cap[p]_{R}:p\in S^{Lmc}\}$, and
$\Gamma\approx S^{Lmc}/R$ if and only if both of the
following statements hold:\\
 (a) Given any topological choice function f for $\Gamma$, there
 is a finite subfamily $\mathcal{F}$ of  $\Gamma$ such that
 $\overline{S}=\bigcup_{\LL\in
 \mathcal{F}}(\overline{f(\LL)})^{\circ}$,\\
 (b) Given distinct $\LL$ and $\K$ in $\Gamma$, there exist
 $A\in \LL^{\circ}$ and $B\in \K^{\circ}$ such that whenever $\C\in
 \Gamma$, for each $H,T\in Z(Lmc(S))$ either
 $\overline{S}-(\overline{A})^{\circ}\subseteq(\overline{H})^{\circ}$then
 $H\in \C$ or
 $\overline{S}-(\overline{B})^{\circ}\subseteq(\overline{T})^{\circ}$then
 $T\in \C$.
 \end{thm}
\begin{proof}
See Theorem 3.9 in \cite{Akbari2}.
\end{proof}
\begin{rem}
i) $\Gamma$ is called a quotient of $S^{Lmc}$ if and only if
$\Gamma$ is a set of pure $z-$filters in $Lmc(S)$ with the
quotient topology satisfying statements $(a)$ and $(b)$ of Theorem
2.8. Let $\Gamma$ be a quotient of $S^{Lmc}$, then by Theorem
2.7, for each $\widetilde{p}\in S^{Lmc}$ there is a unique
$\LL\in\Gamma$ such that $\LL\subseteq \widetilde{p}$. Then $\gamma
:S^{Lmc}\rightarrow \Gamma$ by
$\gamma(\widetilde{p})\subseteq \widetilde{p}$ is a quotient map.
We define $e:S\rightarrow\Gamma$ by $e(s)=\gamma(\widehat{s})$,( see Corollary 3.12 of \cite{Akbari2}.)

ii) Let $\Gamma$ be a quotient of $S^{Lmc}$. If for
 each $\A$ and $\B$ in $\Gamma$ there is some $\C\in \Gamma$
 with $\C\subseteq \A\odot\B$, we define $\dot{+}$ on $\Gamma$ by $\A\dot{+} \B\in
 \Gamma$ and $\A\dot{+} \B\subseteq \A\odot \B$. Then $e$ is a continuous homomorphism from $S$ to $\Gamma$,
 $\Gamma$ is a right topological semigroup, $e[S]$ is dense in $\Gamma$, and
 the function $\lambda_{e(s)}$ is continuous for every $s\in S$,( see Theorem 3.15 in \cite{Akbari2}.)
 \end{rem}
\begin{thm}
Let $\Gamma$ be a set of pure $z-$filters with the quotient
topology. The following statements are equivalent:\\
 (a) There exists a continuous function
$h:S\longrightarrow \Gamma$ such that $\Gamma$ is a Hausdorff
compact space, $h[S]$ is dense in $\Gamma$ and

 (i) for each $s\in
S$, $s\in\bigcap h(s)$ and

(ii) for distinct $\LL$ and $\K$ in $\Gamma$, there exists $B\in
\K^{\circ}$ such that for each  $A\in Z(Lmc(S))$ if
$\overline{S}-(\overline{B})^{\circ}\subseteq
(\overline{A})^{\circ}$ then $A\in\LL$.\\
(b) $\Gamma$ is quotient of $S^{Lmc}$.\\
(c) $\Gamma$ is quotient of $S^{Lmc}$, $\Gamma$ is a
Hausdorff compact space and $e[S]$ is dense in $\Gamma$.
\end{thm}

\begin{proof}
 See Theorem 3.16 in \cite{Akbari2}.
 \end{proof}

\section{\bf Weakly Almost Periodic Compactificatin as a Quotient Space}

 In this section, $S^{wap}$ is described as a space of pure
$z-$filters. We obtain, in Theorem 3.13, a description of $S^{wap}$
as a space of pure filters which is internal to the set of pure filters.
In this section assume
$\widehat{f}=(\varepsilon^{\ast})^{-1}(f)\in\mathcal{C}(S^{Lmc})$ for $f\in Lmc(S)$.
\begin{lem}
Let $\Gamma$ be a quotient of $S^{Lmc}$, $T$ be a
compact Hausdorff space, and let $f:S\rightarrow T$ be a continuous
function with a continuous extension $\widetilde{f}:S^{Lmc}\rightarrow T$. Statements $(a)$ and $(d)$ are equivalent. If $\{
\hat{s}:s\in S\}\subseteq \Gamma$, then all four statements
are equivalent:\\
$(a)$ f has a continuous extension to $\Gamma$.(that is, there
exists $g:\Gamma\rightarrow T$ such that $g\circ e=f$).\\
$(b)$ For each $\LL\in \Gamma$,
$\bigcap\{cl_{T}(f[A]):\overline{\LL}-\overline{A^{c}}\neq\emptyset ,A\in
Z(Lmc(S))\}\neq \emptyset$.\\
$(c)$ For each $\LL\in \Gamma$, $card(\bigcap\{cl_{T}f[A]:
\overline{\LL}-\overline{A^{c}}\neq\emptyset,A\in
Z(Lmc(S))\})=1$.\\
$(d)$ For each $\LL\in\Gamma$, $\widetilde{f}$ is constant on
$\overline{\LL}$.
\end{lem}

\begin{proof}
 To see that $(a)$ implies $(d)$, note that
$g\circ\gamma$ is continuous extension of $f$ and hence
$g\circ\gamma=\widetilde{f}$.\\
To see that $(d)$ implies $(a)$, define $g$ by
$g(\gamma(\widetilde{p}))=\widetilde{f}(\widetilde{p})$. Since
$\widetilde{f}$ constant on $\overline{\LL}$ for each $\LL\in \Gamma$,
$\,\,g$ is well defined . Therefore $g\circ\gamma=\widetilde{f}$ and
$\gamma$ is a quotient map. This implies $g$ is continuous.\\
Now assume that $\{\hat{s}:s\in S\}\subseteq \Gamma$ so that for
all $s\in S$, $e(s)=\hat{s}$. That $(c)$ implies $(b)$ is trivial.\\
To see that $(a)$ implies $(c)$, we show that
\begin{equation*}  \bigcap\{cl_{T}f(A):
\overline{\LL}-\overline{A^{c}}\neq\emptyset ,A\in
Z(Lmc(S))\}=\{g(\LL)\}.\end{equation*}

Since below diagram commutes,
\[
\begin{matrix}
S &\stackrel{f}{ \rightarrow} & T\\
\stackrel{\varepsilon}{ \downarrow} & \,& \stackrel{g}{\uparrow}\\
S^{Lmc} &\stackrel{\gamma}{ \rightarrow}& \Gamma, \\
\end{matrix}
\]
($i.e.\ f=g\circ\gamma\circ \varepsilon$).\\
Hence for each $A\in Z(Lmc(S))$, we have
\begin{eqnarray*}
cl_{_{T}}f(A)&=& cl_{T}g\circ\gamma\circ\varepsilon(A)=
g\circ\gamma(\overline{\varepsilon(A)})\\
&=&g\circ\gamma(\overline{A})= g(\gamma(\overline{A})).
\end{eqnarray*}

Also if $\overline{\LL}-\overline{A^{c}}\neq\emptyset$, then there exists
$\widetilde{p}\in \overline{\LL}\cap\overline{A}$, ( it is obvious that  $\overline{A}^{c}\subseteq
A^{c}\subseteq \overline{A^{c}}$ hence if $\widetilde{p}\notin
\overline{A^{c}}$ then $\widetilde{p}\in \overline{A}$). So
\begin{equation*}  cl_{T}f(A)=g(\gamma(\overline{A}))=\{g(\gamma(\widetilde{r})):\widetilde{r}\in
\overline{A}\}.\end{equation*}
 Since $\widetilde{p}\in \overline{A}\bigcap
\overline{\LL}$ so $g(\LL)=g(\gamma(\widetilde{p}))\in
cl_{T}f(A)$. Also for each $\K\in \Gamma$ if $\K\neq \LL$ then there
exists $A\in Z(Lmc(S))$ such that $\overline{A}\bigcap
\overline{\K}=\emptyset$, and this conclude  that $g(\K)\notin cl_{T}f(A)$.
Hence
\begin{equation*}  \bigcap \{cl_{T}f(A): \overline{\LL}-\overline{A^{c}}\neq
\emptyset\}=\{g(\LL)\}.\end{equation*}

To see that $(b)$ implies $(a)$. Choose for each $\LL\in \Gamma$
some $g(\LL)\in \bigcap \{cl_{T}f(A):
\overline{\LL}-\overline{A^{c}}\neq \emptyset\}$. It is obvious that
$g\circ e =f$. To see that $g$ is continuous, we show that
$\widetilde{f}$ is constant on $\overline{\LL}$, for each $\LL\in
\Gamma$. Let $\p\in \overline{\LL}$ and $A\in Z(Lmc(S))$,
if $\p\in \overline{A}$ then $\overline{\LL}- \overline{A^{c}}\neq
\emptyset$ and so $g(\LL)\in cl_{T}f(A)$. Since
$cl_{T}f(A)=\widetilde{f}(\overline{A})$ so $\bigcap_{\p\in
\overline{A}}\widetilde{f}(\overline{A})= \{\widetilde{f}(\p)\}$. This
conclude that $g(\LL)=\widetilde{f}(\p)$ and
$g\circ\gamma=\widetilde{f}$. $\gamma$ is a quotient map and
$\widetilde{f}$ is continuous hence $g$ is continuous.
\end{proof}
\begin{lem}
Let $\Gamma$ be a quotient of $S^{Lmc}$ such that
$\{\hat{s}:s\in S\}\subseteq \Gamma$, $(T,+)$ be a compact
Hausdorff right topological semigroup, and  $f$ be a continuous
homomorphism from S to T such that $l_{f(x)}$ is continuous for each
$x\in S$. Let $\bigcap \{cl_{T}f(A):
\overline{\LL}-\overline{A^{c}}\neq \emptyset,\,\,\, A\in
Z(Lmc(S))\}\neq \emptyset$ for each $\LL\in \Gamma$. Let $\LL\dot{+}\K\in \Gamma$
  for each $\LL,\K\in\Gamma$. Then there is a
continuous homomorphism $g: \Gamma\rightarrow T$ such that $g\circ
e= f$.
\end{lem}
\begin{proof}
 Since $\overline{f(S)}=H$ is a compact right
topological semigroup, so $(H,f)$ is a semigroup compactification of
S. By Lemma 3.1, $g:\Gamma \rightarrow T$ exists, $g$ is continuous
and $g\circ e= f$. If $\widetilde{f}:S^{Lmc} \rightarrow T$
be extension of $f:S \rightarrow T$, then $\widetilde{f}$ is a
continuous homomorphism and the below diagram commutes,
\[
\begin{matrix}
\, & S^{Lmc} &\stackrel{\gamma}{\longrightarrow} &\, & \Gamma\\
&\stackrel{\widetilde{f}}{\searrow} & \,&\stackrel{g}{\swarrow} &\\
&\, &T & \\
\end{matrix}
\]
(i.e. $g\circ\gamma=\widetilde{f}$).

Hence for each $\LL,\K\in \Gamma$, there exist
$\widetilde{p},\widetilde{q} \in S^{Lmc}$ such that
$\gamma(\widetilde{p})=\LL$ and $\gamma(\widetilde{q})=\K$. Now we
have:
\begin{eqnarray*}
g(\LL\dot{+}\K)&=& g(\gamma(\widetilde{p})\dot{+}\gamma(\widetilde{q}))\\
&=& g(\gamma(\widetilde{p}*\widetilde{q}))\hspace{1cm} \gamma\hspace{.25 cm} \mbox{is a homomorphism}\\
&=& \widetilde{f}(\widetilde{p}*\widetilde{q})\\
&=&\widetilde{f}(\p)+\widetilde{f}(\q)\hspace{1cm} \widetilde{f}\hspace{.25 cm} \mbox{is a homomorphism}\\
&=& g\circ\gamma(\widetilde{p})+ g\circ\gamma(\widetilde{q})\\
&=& \widetilde{f}(\widetilde{p})+ \widetilde{f}(\widetilde{q})\\
&=& g(\LL)+ g(\K),
\end{eqnarray*}
and so $g$ is a homomorphism.
\end{proof}

We use the set $\mathbb{Q}^+$ of positive rational in the definition, but any
other dense subsets of the positive reals would do.
\begin{defn}
$(a)$ $\varphi$ is a topological $W-nest$ in S if and only if\\
(1) $\varphi: \mathbb{Q}^{+}\rightarrow
Z(Lmc(S))$,\\
(2) $S\in Range(\varphi)$ and $\emptyset\notin Range(\varphi)$,\\
(3) if $r,s\in \mathbb{Q}^{+}$ and $r< s$ then
$(\overline{\varphi(r)})^{\circ}\subseteq
\overline{\varphi(r)}\subseteq (\overline{\varphi(s)})^{\circ}$,
and\\
(4) there do not exist r and $\upsilon$ in $\mathbb{Q}^{+}$ and
sequences $<t_{n}>_{n=1}^{\infty}$ and $<s_{n}>_{n=1}^{\infty}$
in S such that $r< \upsilon$ and whenever $k< n$,  $t_{n}s_{k}\in
(\overline{\varphi(r)})^{\circ}$ and $t_{k}s_{n}\notin
\overline{\varphi(\upsilon)}$.\\
$(b)$ $\mathcal{W}(S)$=\{$\varphi:\varphi$ is a topological $W-nest$
in S\}.
\end{defn}

\begin{lem}
(a) If $f\in wap(S)$, $ a\in cl_{\mathbb{C}}f(S)$, and for
 $r\in Q^{+}$, \begin{equation*}  \varphi(r)=\{ s\in S: \mid f(s)-a\mid\leq r\},\end{equation*}
then $\varphi \in\mathcal{W}(S)$. Also if $\widehat{f}(\p)=a$ for some
$\p\in S^{Lmc}$, then $\p\in(\overline{\varphi(r)})^{\circ}$ for each
$r\in \mathbb{Q}$.\\
(b) Let $\varphi\in \mathcal{W}(S)$, then $f(s)=inf\{ r\in
\mathbb{Q}^{+}: s\in\varphi(r)\}$ is a continuous function.\\
(c) If $\varphi\in \mathcal{W}(S)$ and $f$ is defined on S by
$f(s)=inf\{ r\in \mathbb{Q}^{+}:s\in \varphi(r)\}$, then $f\in
wap(S)$. Further, if $\widetilde{p}\in S^{Lmc}$ and
$\widetilde{p}\in \bigcap_{r\in\mathbb{
Q}^{+}}\overline{\varphi(r)}$ then $\widehat{f}(\widetilde{p})=0$.\\
(d) Let $f\in wap(S)$, $ a\in cl_{\mathbb{C}}f(S)$ and $r>0$.
Then
\begin{equation*}  B=\{x\in S:|f(x)-a|\leq r\}\in Z(wap(S)).\end{equation*}
 Also let
$\widehat{f}(\p)=0$ for some $\p\in S^{Lmc}$. Then
$\p\in (\overline{f^{-1}([0,r])})^{\circ}$ for every $r>0$.
\end{lem}
\begin{proof}
 (a) Condition (3) is immediate, (2) holds since $f$
is bounded. To see that (1), it is obvious that for each $s\in S$,
$g(s)= min(\mid f(s)-a\mid - r, 0)$ is continuous, and
\begin{equation*}  Z(g)=\{ s\in
S: g(s)= 0\}=\{s\in S:\mid f(s)-a\mid \leq r\}.\end{equation*}
Therefore
$\varphi(r)=Z(g)\in Z(Lmc(S))$. To see that (4) holds,
suppose we have such $r,\upsilon,<t_{n}>_{n=1}^{\infty}$ and
$<s_{n}>_{n=1}^{\infty}$. Then the sequences
$<s_{n}>_{n=1}^{\infty}$ and $<t_{n}>_{n=1}^{\infty}$ so that
$lim_{n\rightarrow\infty}lim_{k\rightarrow\infty}f(t_{n}s_{k})$ and
$lim_{k\rightarrow\infty}lim_{n\rightarrow\infty}f(t_{n}s_{k})$
exist. Then $lim_{n\rightarrow\infty}lim_{k\rightarrow\infty}\mid
f(t_{n}s_{k})-a\mid\geq \upsilon$ while
$lim_{k\rightarrow\infty}lim_{n\rightarrow\infty}\mid
f(t_{n}s_{k})-a\mid\leq r$, a contradiction.\\
If there exists $r\in \mathbb{Q}$ such that
$\p\notin(\overline{\varphi(r)})^{\circ}$ then for every
$s\in\mathbb{Q}\cap (0,r)$ we have $\p\notin(\overline{\varphi(s)})$
and this is a contradiction.

 (b) Let there exists a
net $\{s_{\alpha}\}$ in $S$ such that $lim_{\alpha}s_{\alpha}=s$ in $S$ and $lim_{\alpha}f(s_{\alpha})\neq f(s)$. By thinning,
there exists $\epsilon_{\circ}>0$ such that
\begin{equation*}
 (f(s)-\epsilon_{\circ},f(s)+\epsilon_{\circ})\bigcap
\{f(s_{\alpha}):\alpha>\beta\}=\emptyset
\end{equation*}
 for some $\beta$. Pick $r_{1},r_{2}\in Q^{+}$ such that
\[
(f(s)-\epsilon_{\circ}<r_{1}<f(s)<r_{2}<f(s)+\epsilon_{\circ}),
\]
 so $\varphi(r_{1})\subseteq \varphi(r_{2})$,
$\varphi(r_{1})\neq \varphi(r_{2})$ and $U=\varepsilon^{-1}((\overline{\varphi(r_{1})})^{c}\bigcap
(\overline{\varphi(r_{2})})^{\circ})$ is a non empty open set in
$S$. Since $s\in U$ so there exists $\eta$ such that for
each $\alpha>\eta$, $f(s_\alpha)\in U$ and so
$r_{1}<f(s_{\alpha})<r_{2}$ and this is a contradiction.

 (c) By (b), $f$ is defined every where and is bounded and
continuous. Suppose $f\notin wap(S)$, so pick sequences
$<t_{n}>_{n=1}^{\infty}$ and $<s_{n}>_{n=1}^{\infty}$ such that
\begin{equation*}  lim_{n\rightarrow\infty}lim_{k\rightarrow\infty}f(t_{n}s_{k})=x\end{equation*}
and
$lim_{k\rightarrow\infty}lim_{n\rightarrow\infty}f(t_{n}s_{k})=y$
with $x>y$. Let $\epsilon=\frac{x-y}{3}$. By thinning the
sequences we may assume that whenever $k<n,\,\,
f(t_{n}s_{k})<y+\epsilon$ and $f(t_{k}s_{n})>x-\epsilon$. Pick
$r<\upsilon$ in $Q^{+}$ such that
$y+\epsilon<r<\upsilon<x-\epsilon$. Then for $k<n,\,\,\,
t_{n}s_{k}\in \varphi(r)$ and $t_{n}s_{k}\notin
\varphi(\upsilon)$. Hence $\varphi\notin\mathcal{W}(S)$, and this
is a contradiction. The last conclusion holds since each
neighborhood of $\widetilde{p}$ includes points $\hat{s}$ with
$f(s)$ arbitrarily close to $0$.

(d) It is obvious.
\end{proof}

\begin{defn}
Define an equivalence relation on $S^{Lmc}$ by
$\widetilde{p}\equiv \widetilde{q}$ if and only if
$\widehat{f}(\widetilde{p})=\widehat{f}(\widetilde{q})$ whenever
$f\in wap(S)$.
\end{defn}

We observe that $\equiv$ is trivially an equivalence relation on
$S^{Lmc}$ and let $[\p]$ be the equivalence class
of $\p\in S^{Lmc}$.
 We shall be interested
in the $z-$filters $\bigcap[\widetilde{p}]$ where $\widetilde{p}\in
S^{Lmc}$. Consequently, we want to know when $A\in
\bigcap[\widetilde{p}]$ in terms of $\widetilde{p}$.

\begin{lem}
Let $\widetilde{p}\in S^{Lmc}$ and let $A\in
Z(Lmc(S))$. The following statements are equivalent.\\
(a) $[\widetilde{p}]\subseteq (\overline{A})^{\circ}$.\\
(b) There exist $f\in wap(S)$ and $\delta>0$ such that $\{s\in
S:\mid f(s)-\widehat{f}(\widetilde{p})\mid\leq\delta\}\subseteq
A$.\\
(c) There exist $f\in wap(S)$ and $\delta>0$ such that
$Range(f)\subseteq [0,1],\widehat{f}(\widetilde{p})=0$, and
$f^{-1}([0,\delta])\subseteq A$.
\end{lem}
\begin{proof}
 That (c) implies (b) is trivial. To see that (b)
implies (a), let $\widetilde{q}\in [\widetilde{p}]$. Then
$\widehat{f}(\widetilde{p})=\widehat{f}(\widetilde{q})$ so $B=\{s\in
S:\mid f(s)-\widehat{f}(\widetilde{p})\mid\leq\delta\}\subseteq A$.
Hence $[\widetilde{p}]\subseteq(\overline{A})^{\circ}$.

To see that (a) implies (c), suppose the conclusion fails and let
\begin{equation*}
  \mathcal{G}=\{f\in wap(S): Range(f)\subseteq [0,1]\mbox{ and }
\widehat{f}(\widetilde{p})=0\}.
\end{equation*}
For each $f\in \mathcal{G}$ and
each $\delta>0$, let
$B(f,\delta)=f^{-1}([0,\delta])-\overline{A}$ and let
\begin{equation*}
 \LL=\{B(f,\delta):f\in \mathcal{G} \mbox{ and }\delta>0\}.
 \end{equation*}
Given $f$
and $g$ in $\mathcal{G},\delta>0$ and $\gamma>0$ , let
$\mu=min\{\delta,\gamma\}$ and define h by
$h(s)=max\{f(s),g(s)\}$. Then $Range(h)\subseteq[0,1]$ and by
Lemma 2.3, $h\in wap(S)$. To see that
$\widehat{h}(\widetilde{p})=0$, suppose that $\{s_\alpha\}$ be a net in $S$ such that
$\{s_{\alpha}\}$ converges to $\widetilde{p}$ in $S^{Lmc}$, since $h(s)=\frac{\mid
f(s)-g(s)\mid}{2}+\frac{f(s)-g(s)}{2}$ so
\begin{eqnarray*}
\widehat{h}(\widetilde{p})&=& lim_{\alpha}h(s_{\alpha})\\
&=&lim_{\alpha}\frac{\mid f(s_{\alpha})-g(s_{\alpha})\mid}{2}+\frac{f(s_{\alpha})-g(s_{\alpha})}{2}\\
&=& \frac{\mid
\widehat{f}(\widetilde{p})-\widehat{g}(\widetilde{p})\mid}{2}+
\frac{\widehat{f}(\widetilde{p})-\widehat{g}(\widetilde{p})}{2}\\
&=&0.
\end{eqnarray*}
Thus $h\in \mathcal{G}$ and $B(h,\mu)\subseteq B(f,\delta)\bigcap
B(g,\gamma)$. Therefore $\LL$ has the finite intersection property. Pick $\widetilde{r}\in S^{Lmc}$ such that
$\widetilde{r}\in\bigcap_{f\in\mathcal{G},\,\delta>0}\overline{B(f,\delta)}$.
Then $\widetilde{r}\notin (\overline{A})^{\circ}$ and so
$\widetilde{r}\notin [\widetilde{p}]$. Pick $g\in wap(S)$ such
that $\widehat{g}(\widetilde{p})\neq \widehat{g}(\widetilde{r})$.
Let
\begin{equation*}
 b=sup\{ \mid
\widehat{g}(\widetilde{q})-\widehat{g}(\widetilde{p})\mid
:\widetilde{q}\in S^{Lmc}\},
\end{equation*}
and pick $n\in \mathbb{N}$
such that $n>b$. Define $\varphi$ as in Lemma 3.4 with
$a=\widehat{g}(\widetilde{p})$. Then $\varphi\in
\mathcal{W}(S)$ and $\varphi(n)=S$. Define $\mu(t)$ for $t\in
Q^{+}$ by $\mu(t)=\varphi(nt)$. Then trivially $\mu\in \mathcal{W}(S)$. Define $f$
on $S$ by
\begin{eqnarray*}
f(s)&=& inf\{t\in Q^{+}:s\in \mu(t)\}\\
&=& inf\{ t\in Q^{+}:s\in \varphi(nt)\}\\
&=& inf\{t\in Q^{+}:\mid g(s)-a \mid\leq nt\}\\
&=& inf\{ t\in Q^{+}:\mid \frac{1}{n}g(s)-\frac{1}{n}a \mid\leq
t\}\\
&=& \mid \frac{1}{n}g(s)-\frac{1}{n}a\mid.
\end{eqnarray*}
By Lemma 3.4, $f\in wap(S)$. Since $\mu(1)=S$, $Range(f)\subseteq [0,1]$. Given $t\in Q^{+}$, $\mu(t)=\{ s\in S: \mid
g(s)-\widehat{g}(\widetilde{p})\mid\leq nt\}$ so
$\widetilde{p}\in\overline{\mu(t)}$. Thus
$\widehat{f}(\widetilde{p})=0$. Also
\begin{equation*}
 \widehat{f}(\widetilde{r})=\mid\widehat{g}(\widetilde{r})-\widehat{g}(\widetilde{p})\mid/n>0.
 \end{equation*}
Thus $
\widetilde{r}\notin\overline{B(f,\widehat{f}(\widetilde{r})/2)}$ is
a contradiction.
\end{proof}

\begin{thm}
Let $\Gamma$ be a set of pure $z-$filters in $Lmc(S)$.
Statement (a) below is equivalent to the conjunction of statements
(b),(c) and (d) and implies that $\Gamma$ is a Hausdorff compact
space and $e[S]$ is dens in $\Gamma$.\\
(a) $\Gamma=\{\bigcap [p]:p\in S^{Lmc}\}$.\\
(b) For $\varphi\in \mathcal{W}(S)$ and $\mathcal{L}\in \Gamma$, if
$\mathcal{L}\bigcup Range(\varphi)$ has the finite intersection
property, then $Rang(\varphi)\subseteq \mathcal{L}$.\\
(c) Given distinct $\mathcal{L}$ and $\K$ in $\Gamma$, there exists
$\varphi\in \mathcal{W}(S)$ such that $Range(\varphi)\subseteq
\mathcal{L}$ but $Range(\varphi)\setminus \K\neq \emptyset$.\\
(d) For each topological choice function $f$ for $\Gamma$ , there is
a finite subfamily $\mathcal{F}$ of $\Gamma$ such that
$\overline{S}=S^{Lmc}=\bigcup_{\mathcal{L}\in
\mathcal{F}}(\overline{f(\mathcal{L})})^{\circ}$.
\end{thm}
\begin{proof}
We show first that (a) implies (b),(c) and (d), so
we assume (a) holds. Observe that if
$\widetilde{p},\widetilde{q}\in S^{Lmc}$ and
$\widetilde{p}\notin [\widetilde{q}]$, then for some $f\in
wap(S),\,\,\,\widehat{f}(\widetilde{p})\neq\widehat{f}(\widetilde{q})$.
Since $\widehat{f}$ is continuous, there is a neighborhood of
$\widetilde{p}$ missing $[\widetilde{q}]$, that is, each
$[\widetilde{q}]$ is closed in $S^{Lmc}$. Thus by
Theorem 2.7 condition (d) holds.

To establish (b), let $\varphi\in \mathcal{W}(S)$ and let
$\mathcal{L}\in \Gamma$ and assume $\mathcal{L}\bigcup
Range(\varphi)$ has the finite intersection property. Pick
$\widetilde{p}\in S^{Lmc}$ such that pure $z-$filter
generated by $\mathcal{L}\bigcup Range(\varphi)$ contained in
$\widetilde{p}$. Let $\K=\bigcap[\widetilde{p}]$. Suppose that
$Range(\varphi)- \K\neq \emptyset$,  pick $\widetilde{q}\in
[\widetilde{p}]$ and $r\in Q^{+}$ such that $\widetilde{q}\notin
\overline{\varphi(r)}$. Define $f$ as in Lemma 3.4(c), then
$\widehat{f}(\widetilde{p})=0$ while $\widehat{f}(\widetilde{q})\geq
r$, a contradiction.

 To establish
(c), let $\mathcal{L}$ and $\K$ be distinct members of $\Gamma$
and pick $\widetilde{p}$ and $\widetilde{q}$ in
$S^{Lmc}$ such that
$\mathcal{L}=\bigcap[\widetilde{p}]$ and
$\K=\bigcap[\widetilde{q}]$. Pick $f\in wap(S)$ such that
$\widehat{f}(\widetilde{p})\neq \widehat{f}(\widetilde{q})$ and
let $a=\widehat{f}(\widetilde{p})$. Define $\varphi$ as in Lemma
3.4. Let
$\epsilon=\mid\widehat{f}(\widetilde{p})-\widehat{f}(\widetilde{q})\mid$
and pick $\nu\in Q^{+}$ such that $\nu<\epsilon$. Then
$\widetilde{q}\notin (\overline{\varphi(\nu)})^{\circ}$ so $\varphi(\nu)\notin \K$. Further,
if $\widetilde{r}\in [\widetilde{p}]$, then
$\widehat{f}(\widetilde{r})=a$. Since
$\rr\in(\overline{\varphi(\nu)})^{\circ}$ for every
$\nu\in\mathbb{Q}^{+}$ therefore $\varphi(\nu)\in\rr$, \cite[Lemma
2.8]{Akbari2} , and so $Range(\varphi)\subseteq \widetilde{r}$. Therefore
$Range(\varphi)\subseteq \mathcal{L}$.

Now we assume that (b),(c) and (d) hold and prove (a). Since (d)
holds, by Theorem 2.7 each $\widetilde{p}\in S^{Lmc}$
contains some $\mathcal{L}\in \Gamma$. It suffices, therefore, to
show that if $\widetilde{p}\in S^{Lmc}$ and
$\mathcal{L}\in \Gamma$ with $\mathcal{L}\subseteq \widetilde{p}$,
then $\mathcal{L}=\bigcap[\widetilde{p}]$. To this end , let
$\widetilde{p}\in S^{Lmc}$ and let $\mathcal{L}\in \Gamma$
with $\mathcal{L}\subseteq \widetilde{p}$. Let $\widetilde{q}\in
[\widetilde{p}]$ and suppose $\mathcal{L}\setminus \widetilde{q}\neq
\emptyset$. By condition (d), pick $\K\in \Gamma$ such that $\K\subseteq
q$. Pick $\varphi$ as guaranteed by condition (c). Note that if we
had $Range(\varphi)\subseteq \widetilde{q}$, we would have
$\K\bigcup Range(\varphi)$ having the finite intersection property
and would thus have $Range(\varphi)\subseteq \K$, by condition (b)
and this is a contradiction with (c). Thus we pick $r\in Q^{+}$ such
that $\widetilde{q}\notin \overline{\varphi(r)}$. Define $f$ as in
Lemma 3.4. Then $\widehat{f}(\widetilde{q})\geq r$ while
$\widehat{f}(\widetilde{p})=0,a$ contradiction . Thus
$\mathcal{L}\subseteq \bigcap [\widetilde{p}]$. To see that $\bigcap [\widetilde{p}]\subseteq\mathcal{L}$, let
$\overline{\bigcap [\widetilde{p}]}\subseteq
(\overline{A})^{\circ}$ and suppose that
$\overline{\LL}-(\overline{A})^{\circ}\neq\emptyset$. Pick
$\widetilde{q}\in S^{Lmc}$ such that $\LL\bigcup\{B\in
Z(Lmc(S)):
\overline{S}-(\overline{A})^{\circ}\subseteq(\overline{B})^{\circ}\}\subseteq
\widetilde{q}$ and pick $f\in wap(S)$ such that
$\widehat{f}(\widetilde{p})\neq\widehat{f}(\widetilde{q})$. Let
$a=\widehat{f}(\widetilde{p})$ and define $\varphi$ as in Lemma
3.4(c). Then $Range(\varphi)\subseteq\widetilde{p}$, by Lemma
3.4(a), so $\LL\bigcup Range (\varphi)$ has the finite
intersection property so by condition (b),
$Range(\varphi)\subseteq\LL$. But then
$Range(\varphi)\subseteq\widetilde{q}$ so
$\widehat{f}(\widetilde{p})=\widehat{f}(\widetilde{q})$, a
contradiction.

Now assume that (a) holds. To show that $\Gamma$ is a Hausdorff
compact space and $e[S]$ is dense in $\Gamma$ is suffices, by Theorem 2.10, to show that
$\Gamma$ is a quotient of $S^{Lmc}$, that is, that
conditions (a) and (b) of Theorem 2.8 hold. Since condition (a)
of Theorem 2.8 and condition (d) of this Theorem are identical,
it suffices to establish condition (b) of Theorem 2.8. To this
end, let $\LL$ and $\K$ be distinct members of $\Gamma$ and pick
$\widetilde{p}$ and $\widetilde{q}$ in $S^{Lmc}$ such
that $\LL=\bigcap[\widetilde{p}]$ and
$\K=\bigcap[\widetilde{q}]$. Pick $f\in wap(S)$ such that
$\widehat{f}(\widetilde{p})\neq\widehat{f}(\widetilde{q})$. Let
$\epsilon=\mid\widehat{f}(\widetilde{p})-\widehat{f}(\widetilde{q})\mid$,
let $A=\{s\in S:\mid
f(s)-\widehat{f}(\widetilde{p})\mid\leq\epsilon/3\}$ and $B=\{s\in
S:\mid f(s)-\widehat{f}(\widetilde{q})\mid\leq\epsilon/3\}$. Then $A,B\in Z(Lmc(S))$, $[\widetilde{p}]\subseteq (\overline{A})^{\circ}$ and
$[\widetilde{q}]\subseteq (\overline{B})^{\circ}$ so
$A\in\LL^{\circ}$ and $B\in\K^{\circ}$. Let $\mathcal{C}\in\Gamma$
and pick $r\in S^{Lmc}$ such that $\mathcal{C}=\bigcap[r]$. If $\mid
\widehat{f}(r)-\widehat{f}(\widetilde{p})\mid\leq 2\epsilon/3$,
then $\overline{\mathcal{C}}\subseteq\overline{S-B}\subseteq
\overline{S}-({\overline{B}})^{\circ}\subseteq
(\overline{H})^{\circ}$ implies that $H\in \mathcal{C}$ for each
$H\in Z(Lmc(S))$, and if $\mid
\widehat{f}(r)-\widehat{f}(\widetilde{p})\mid\geq \epsilon/3$ then
$\overline{\mathcal{C}}\subseteq\overline{S-A}\subseteq
\overline{S}-({\overline{A}})^{\circ}\subseteq
(\overline{T})^{\circ}$ implies $T\in \mathcal{C}$ for each $T\in
Z(Lmc(S))$.
\end{proof}

We want to show that
$\Gamma=\{\bigcap[\widetilde{p}]:\widetilde{p}\in
S^{Lmc}\}$ is $S^{wap}$. In order to use Remark 2.9,
we need that for each $\LL$ and $\K$ in $\Gamma$ there exists
$\C\in \Gamma$ with $\C\subseteq\LL+\K$. For such $\C$, we have
$\overline{\LL+\K}\subseteq\overline{\C}$. By \cite[Lemma 2.10]{Akbari2}, we have,
if $\LL=\bigcap[\widetilde{p}]$ and $\K=\bigcap[\widetilde{q}]$,
then $\widetilde{p}*\widetilde{q}\in\overline{\LL+\K}$. Since for
distinct $\C$ and $\D$ in $\Gamma$,
$\overline{\C}\bigcap\overline{\D}=\emptyset$, our only candidate
for $\C$ in $\Gamma$ with $\C\subseteq\LL+\K$ is thus
$\bigcap[\widetilde{p}*\widetilde{q}]$.

\begin{lem}
If for all $\widetilde{p},\widetilde{q}$ and $\widetilde{r}$ in
$S^{Lmc}$, $\widetilde{q}\equiv\widetilde{r}$ implies
$(\widetilde{q}*\widetilde{p})\equiv(\widetilde{r}*\widetilde{p})$,
then for all $\widetilde{p}$ and $\widetilde{q}$ in
$S^{Lmc}$,

\begin{equation*}  \bigcap[\widetilde{q}*\widetilde{p}]
\subseteq\bigcap[\widetilde{q}]+\bigcap[\widetilde{p}].\end{equation*}
\end{lem}
\begin{proof}
Let $\widetilde{p}$ and $\widetilde{q}$ be in
$S^{Lmc}$ and let
$[\widetilde{q}*\widetilde{p}]\subseteq (\overline{A})^{\circ}$.
We show that $A\in
\bigcap[\widetilde{q}]+\bigcap[\widetilde{p}]$. To this end, we
let $\widetilde{r}\in [\widetilde{q}]$ and show that for each
$F\in Z(Lmc(S))$,
$\Omega_{\bigcap[\widetilde{p}]}(A)\subseteq F$ implies
$F\in\widetilde{r}$,(i.e. $A\in \cap[\q]$ ). Since
$\widetilde{r}\in [\widetilde{q}]$ we have by assumption that
$(\rr*\p) \equiv(\q*\p)$. Thus $[\rr*\p]\subseteq
(\overline{A})^{\circ}$. Pick, by Lemma 3.6, $f\in wap(S)$ and
$\delta >0$ such that $Range(f)\subseteq [0,1]$,
$\widehat{f}(\rr*\p)=0$ and $f^{-1}([0,\delta])\subseteq A$. Let
$B=f^{-1}([0,\delta/2])$. Then $B\in \rr+\p$ and so for each $F\in
Z(Lmc(S))$, $\Omega_{\p}(B)\subseteq F$ implies that
$F\in\rr$. For $s\in S$, define $g_{s}$ by
$g_{s}(t)=f(st)=f\circ\lambda_{s}(t)$. Then $g_{s}\in wap(S)$ and
for $\gamma >0$,
$g^{-1}_{s}([0,\gamma])=\lambda^{-1}_{s}(f^{-1}([0,\gamma])$. Now
given $s\in \Omega_{\p}(B)$, we have
$g^{-1}_{s}([0,\delta/2])\in\p$ so
$\widehat{g}_{s}(\p)\leq\delta/2$. Thus
\begin{equation*}
  \{t\in S: \mid
g_{s}(t)-\widehat{g}_{s}(\p)\mid\leq\delta/2\}\subseteq
g^{-1}_{s}([0,\delta/2])
\end{equation*}
 so, by Lemma 3.6,
$g^{-1}_{s}([0,\delta/2])\in\bigcap[\p]$. Since
$g^{-1}_{s}([0,\delta/2])\subseteq
(f\circ\lambda_{s})^{-1}([0,\delta])$ we have
\begin{equation*}
  \Omega_{\p}(B)\subseteq\{s\in
S:(f\circ\lambda_{s})^{-1}([0,\delta])\in\bigcap[\p]\}.
\end{equation*}
 Since $f^{-1}([0,\delta])\subseteq A$,
 \begin{equation*}  \Omega_{\p}(B)\subseteq\{s\in
S:\lambda^{-1}_{s}(A)\in\bigcap[\p]\}.
\end{equation*}
So for each $H\in
Z(Lmc(S))$,
\begin{equation*}  \Omega_{\cap[\p]}(A)=\{s\in
S:\lambda^{-1}_{s}(A)\in\bigcap[\p]\}\subseteq H
\end{equation*}
 implies $H\in
\rr$, and this conclude $H\in\bigcap[\rr]$, hence
$A\in\bigcap[\rr]+\bigcap[\p]$. Now we can conclude
\begin{eqnarray*}
\overline{\bigcap[\rr]+\bigcap[\p]}&=&\bigcap_{B\in\bigcap[\rr]+\bigcap[\p]}(\overline{B})\\
&\subseteq&\bigcap_{[\q*\p]\subseteq(\overline{B})^{\circ}}\overline{B}\\
&=&[\q*\p]
\end{eqnarray*}
 and this implies that
\begin{eqnarray*}
\bigcap[\widetilde{q}*\widetilde{p}]
\subseteq\bigcap[\widetilde{q}]+\bigcap[\widetilde{p}].
\end{eqnarray*}
\end{proof}
\begin{lem}
Let $\p,\q\in S^{Lmc}$ and let $A\in Z(Lmc(S))$.
If $[\q*\p]\subseteq (\overline{A})^{\circ}$, then for each $B\in
Z(Lmc(S))$, $\{t\in S:[\q]\subseteq
\overline{r_{t}^{-1}(A)}\}\subseteq B$ implies $B\in \p$.
\end{lem}
\begin{proof}
 Let $[\q*\p]\subseteq (\overline{A})^{\circ}$, and
pick by Lemma 3.6, $f\in wap(S)$ and $\delta >0$ such that
$Range(f)\subseteq [0,1]$, $\widehat{f}(\q*\p)=0$ and
$f^{-1}([0,\delta])\subseteq A$. Let $B=f^{-1}([0,\delta/3])$.
Then $B\in\bigcap[\q*\p]\subseteq\q*\p$ so, in particular, for
each $T\in Z(Lmc(S))$, $\Omega_{\p}(B)\subseteq T$
implies $T\in\q$.

Given $t\in S$, define $g_{t}\in wap(S)$ by $g_{t}(s)=f\circ
r_{t}(s)=f(st)$. Then for $\gamma>0$,
$g^{-1}_{t}([0,\gamma])=r^{-1}_{t}(f^{-1}([0,\gamma]))$. Note that
if $g^{-1}_{t}([0,\delta/2])\in\q$, then $\widehat{g}_{t}(\q)\leq
\delta/2$ so that
\begin{equation*}
  \{s\in S:\mid
g_{t}(s)-\widehat{g}_{t}(\q)\mid<\delta/2\}\subseteq
g^{-1}_{t}([0,\delta])
\end{equation*}
 and hence by Lemma 3.6,
$g^{-1}_{t}([0,\delta])\in \bigcap[\q]$. Thus it suffices to show,
with $D=f^{-1}([0,\delta/2])$, that for each $H\in
Z(Lmc(S))$,
\begin{equation*}
 \{t\in S: [\q]\subseteq
\overline{r_{t}^{-1}(D)}\}\subseteq H
\end{equation*}
 implies $H\in\p$.( Because
$\{t\in S: [\q]\subseteq \overline{r_{t}^{-1}(D)}\}\subseteq
\{t\in S: \q\in \overline{r_{t}^{-1}(D)}\}$.) Suppose instead
that there exists $H\in Z(Lmc(S))$ such that $\{t\in S:
\q\in \overline{r_{t}^{-1}(D)}\}\subseteq H$ and $H\notin\p$.
Hence there exists an $z-$ultrafilter $p\subseteq
Z(Lmc(S))$  such that $H\notin p$ and by \cite[Lemma 2.5(10)]{Akbari2}, for each $F\in Z(Lmc(S))$, $H^{c}\subseteq F$
implies that $F\in p$. So for each $F\in Z(Lmc(S))$,
$H^{c}\subseteq\{t\in S:\q\notin
\overline{r_{t}^{-1}(D)}\}\subseteq F$ implies that $F\in p$.

 Let
$t_{1}\in E=\{t\in S:\q\notin \overline{r_{t}^{-1}(D)}\}$.
Inductively pick
\begin{eqnarray*}
s_{n}&\in&
 r _{\p}^{-1}((\overline{B})^{\circ})\cap\bigcap_{k=1}^{n}(S-r^{-1}_{t_{k}}(D))
\end{eqnarray*}
(if $ r
_{\p}^{-1}((\overline{B})^{\circ})\cap\bigcap_{k=1}^{n}(S-r^{-1}_{t_{k}}(D))=\emptyset$
then
\begin{equation*}
 \varepsilon^{-1}( r
_{\p}^{-1}((\overline{B})^{\circ}))\subseteq\bigcup_{k=1}^{n}r^{-1}_{t_{k}}(D)
\end{equation*}
and so $\q\in\overline{\bigcup_{k=1}^{n}r^{-1}_{t_{k}}(D)}=
\bigcup_{k=1}^{n}\overline{r^{-1}_{t_{k}}(D)}$  is a contradiction.) and
\begin{eqnarray*}
 t_{n+1}&\in& E\cap \bigcap_{k=1}^{n}(\overline{\lambda_{s_{k}}^{-1}(B)})^{\circ}.
\end{eqnarray*}
($E\cap
\bigcap_{k=1}^{n}(\overline{\lambda_{s_{k}}^{-1}(B)})^{\circ}\neq\emptyset$,
because $\p$ is a closure point of $E$ and interior point of
$\bigcap_{k=1}^{n}(\overline{\lambda_{s_{k}}^{-1}(B)})^{\circ}$.)

Then, if $k\leq n$ we have $s_{n}t_{k}\notin D$ so that
$f(s_{n}t_{k})\geq\delta/2$. Also, if $n<k$ then $s_{n}t_{k}\in B$
so that $f(s_{n}t_{k})\leq\delta/3$. Thus, thinning the sequence
$<t_{n}>_{n=1}^{\infty}$ and $<s_{n}>_{n=1}^{\infty}$ so that all
limits exist, we have
$lim_{k\rightarrow\infty}lim_{n\rightarrow\infty}f(s_{n}t_{k})\geq\delta/2$
while
$lim_{n\rightarrow\infty}lim_{k\rightarrow\infty}f(s_{n}t_{k})\leq\delta/3$
so that $f\notin wap(S)$.\end{proof}

\begin{thm}
For all $\p$ and $\q$ in $S^{Lmc}$
\begin{equation*}  \bigcap[\widetilde{q}*\widetilde{p}]
\subseteq\bigcap[\widetilde{q}]+\bigcap[\widetilde{p}].\end{equation*}
\end{thm}
\begin{proof} Let $\p,\q$ and $\rr$ be in $S^{Lmc}$
and assume $\q \equiv \rr$. By Lemma 3.8 it suffices to show that
$(\q*\p)\equiv(\rr*\p)$. Suppose instead we have $f\in wap(S)$
such that $\widehat{f}(\q*\p)\neq\widehat{f}(\rr*\p)$ and let
$\epsilon =\mid\widehat{f}(\q*\p)-\widehat{f}(\rr*\p)\mid$. Let
\begin{equation*}  A=\{s\in S:\mid{f}(s)-\widehat{f}(\q*\p)\mid\leq\epsilon/3\}\end{equation*}
and let
\begin{equation*}  B=\{s\in
S:\mid{f}(s)-\widehat{f}(\rr*\p)\mid\leq\epsilon/3\}.\end{equation*}
By Lemma
3.6, $[\q*\p]\subseteq(\overline{A})^{\circ}$ and
$[\rr*\p]\subseteq(\overline{B})^{\circ}$, in particular
$B\in\rr*\p$. By Lemma 3.9, for each $T\in Z(Lmc(S))$,
\begin{equation*}  \{t\in S: [\q]\subseteq\overline{r_{t}^{-1}(A)}\}\subseteq T\end{equation*}
implies that $T\in\p$. Since $[\q]=[\rr]$, we have for each $T\in
Z(Lmc(S))$,
\begin{equation*}  \{t\in S:
\rr\in\overline{r_{t}^{-1}(A)}\}\subseteq T\end{equation*}
implies that
$T\in\p$. Let $C=\varepsilon^{-1}( r
_{\p}^{-1}((\overline{B})^{\circ}))\subseteq \Omega_{\p}(B)$ (see
\cite[Lemma 2.7]{Akbari2},) and let $D=\{t\in S:
\rr\in\overline{r_{t}^{-1}(A)}\}$. Let $t_{1}\in D$. Inductively
let $s_{n}\in C\cap(\bigcap_{k=1}^{n}r_{t_{k}}^{-1}(A))$ and let
$t_{n+1}\in
D\cap(\bigcap_{k=1}^{n}(\overline{\lambda_{s_{k}}^{-1}(B)})^{\circ})$.
Thinning the sequence $<t_{n}>_{n=1}^{\infty}$ and
$<s_{n}>_{n=1}^{\infty}$ we get

\begin{eqnarray*}
\mid lim_{n\rightarrow\infty}lim_{k\rightarrow\infty}f(s_{k}t_{n})
 -\widehat{f}(\q*\p)\mid\leq\epsilon/3
 \end{eqnarray*}
and
\begin{eqnarray*}
\mid lim_{k\rightarrow\infty}lim_{n\rightarrow\infty}f(s_{k}t_{n})
 -\widehat{f}(\rr*\p)\mid\leq\epsilon/3,
 \end{eqnarray*}
 a contradiction.\end{proof}

 \begin{thm}
Let $\Gamma =\{\bigcap[\p]:\p\in S^{Lmc}\}$. Then $\Gamma$
is a semitopological semigroup.
 \end{thm}

 \begin{proof} Let $\p\in S^{Lmc}$ and let
 $\LL=\bigcap[\p]$. We need only show $\lambda_{\LL}$ is continuous.
 Suppose not and pick $\q\in S^{Lmc}$ such that
 $\lambda_{\LL}$ is not continuous at $\bigcap[\q]$, and pick
 $A\in\bigcap[\p*\q]$ such that $[\p*\q]\subseteq
 (\overline{A})^{\circ}$ and $\lambda_{\LL}^{-1}(\gamma
 ((\overline{A})^{\circ}))$ contains no neighborhood of $\bigcap
 [\q]$. Pick $\delta >0$ and $f\in wap(S)$ such that
  $Range(f)\subseteq [0,1]$, $\widehat{f}(\p*\q)=0$ and $f^{-1}([0,\delta])\subseteq
  A$.

  We show first that for all $D\in (\bigcap[\q])^{\circ}$ there
  is some $t\in D$ such that
  \begin{equation*}  \p\in\overline{\{s\in S: f(st)\geq\delta/2\}}.\end{equation*}
  Suppose instead we have $D\in(\bigcap[\q])^{\circ}$ such that for
  all $t\in D$,
  \begin{equation*}  \p\notin\overline{\{s\in S: f(st)\geq\delta/2\}}\end{equation*}
  and so by \cite[Lemma 2.5(12)]{Akbari2}, for
  all $t\in D$ and for every $F\in Z(Lmc(S))$,
  \begin{equation*}  \{s\in S: f(st)<\delta/2\}\subseteq F\end{equation*}
   implies
  that $F\in\p$. Since $\gamma((\overline{D})^{\circ})$ is not
  contained in $\lambda_{\LL}^{-1}(\gamma((\overline{A})^{\circ}))$,
  pick $\rr\in S^{Lmc}$ such that $D\in
  (\bigcap[\rr])^{\circ}$ and $A\notin \bigcap[\p*\rr]$. Now if $\widehat{f}(\p*\rr)<\delta$, we would have, with
  $\mu=\delta-\widehat{f}(\p*\rr)$,
  \begin{equation*}  \{s\in S:\mid f(s)-\widehat{f}(\p*\rr)\mid<\mu\}\subseteq A\end{equation*}
   and
  hence $A\in\bigcap[\p*\rr]$ by Lemma 3.6. Thus
  $\widehat{f}(\p*\rr)\geq\delta$. Thus
  \begin{equation*}  \{s\in S: f(st)\geq
  2\delta/3\}\in\p*\rr.\end{equation*}
  Let $B=\{s\in S: f(st)\geq
  2\delta/3\}$ and let
  \begin{equation*}  C=\varepsilon^{-1}( r _{\widetilde{r}}^{-1}
  ((\overline{B})^{\circ}))\subseteq
  \Omega_{\rr}(B)=\{s\in S:
  \lambda_{s}^{-1}(B)\in\rr\}.\end{equation*}
   Then for each $F\in
  Z(Lmc(S))$, $C\subseteq F$ implies that $F\in\p$. Let $s_{1}\in
  C$, and, inductively, let
  \begin{equation*}  t_{n}\in D\cap (
  \bigcap_{k=1}^{n}(\lambda_{s_{k}}^{-1}(B))),\end{equation*}
   and let
  \begin{equation*}  s_{n+1}\in C\cap (\bigcap_{k=1}^{n}\{s\in
  S:f(st_{k})<\delta/2\}).\end{equation*}
  Then, after thinning we have
  \begin{equation*}  lim_{n\rightarrow\infty}lim_{k\rightarrow\infty} f(s_{k}t_{n})\leq
  \delta/2\end{equation*}
   while
  \begin{equation*}  lim_{k\rightarrow\infty}lim_{n\rightarrow\infty} f(s_{k}t_{n})\geq
  2\delta/3.\end{equation*} This contradiction establishes the claim.

  Let $E=f^{-1}([0,\delta/3])$ and let
  $F=\varepsilon^{-1}( r _{\cap[\q]}^{-1}((\overline{E})^{\circ}))$. Then
  $E\in\cap[\p*\q]$ so, by Theorem 3.10, for each $T\in
  Z(Lmc(S))$, $F\subseteq T$ implies that $T\in\cap[\p]$ and
  hence $T\in \p$. Pick $s_{1}\in F$. Inductively,
  $\cap_{k=1}^{n}\lambda_{s_{k}}^{-1}(E)\in (\cap[\q])^{\circ}$ so pick
  $t_{n}\in\cap_{k=1}^{n}\lambda_{s_{k}}^{-1}(E)$ such that $\p\in\overline{\{s\in
  S:f(st_{n})\geq\delta/2\}}$. Pick
  \begin{equation*}  s_{n+1}\in F\cap\bigcap_{k=1}^{n}\{s\in  S:f(st_{k})\geq\delta/2\}.\end{equation*}
   Again, after thinning we obtain
  \begin{equation*}  lim_{n\rightarrow\infty}lim_{k\rightarrow\infty} f(s_{k}t_{n})\geq
  \delta/2\end{equation*}
    while
  $lim_{k\rightarrow\infty}lim_{n\rightarrow\infty} f(s_{k}t_{n})\leq
  \delta/3.$ \end{proof}

  The next theorem says that $\{\cap[\p]:\p\in S^{Lmc}\}$
  is $S^{wap}$.

  \begin{thm}
  Let $\Gamma=\{\cap[\p]:\p\in S^{Lmc}\}$. Then\\
  (1) $e$ is a continuous homomorphism from $S$ to $\Gamma$,\\
  (2) $\Gamma$ is a compact Hausdorff semitopological semigroup,\\
  (3) $e(S)$ is dense in $\Gamma$, and\\
  (4) if $(T,\varphi)$ satisfies (1) and (2), with $T$ replacing
  $\Gamma$ and $\varphi$ replacing $e$, there is a continuous
  homomorphism $\eta :\Gamma\rightarrow T$ such that $\eta\circ
  e=\varphi$.
  \end{thm}

   \begin{proof}
   Statements (1), (2) and (3) follow from Remark 2.9, Theorem
   3.7, 3.10 and 3.11. (Theorem 3.10 need to show that the hypotheses
   of Remark 2.9 are satisfied.)

   Let $(T,\varphi)$ satisfy (1) and (2). We first show there is a
   continuous $\eta: \Gamma\rightarrow T$ such that $\eta\circ
   e=\varphi$. By Lemma 3.1 it suffices to show for this that for
   each $\p\in S^{Lmc}$, $\widetilde{\varphi}$ is constant
   on $[\p]$, where $\widetilde{\varphi}:S^{Lmc}\rightarrow T$ is the continuous extension of
   $\varphi :S\rightarrow T$. Suppose instead that we have $\p \equiv\q$ with
   $\widetilde{\varphi}(\p)\neq\widetilde{\varphi}(\q)$. Since $T$ is
   completely regular, pick $f\in \mathcal{C}(T)$ such that
   $f(\widetilde{\varphi}(\p))\neq f(\widetilde{\varphi}(\p))$. By Lemma
   2.3(ii), $f\in wap(T)$. Consequently
   $f\circ\varphi\in wap(S)$.(Given $<t_{n}>_{n=1}^{\infty}$ and
    $<s_{n}>_{n=1}^{\infty}$ in $S$ such that
     $lim_{n\rightarrow\infty}lim_{k\rightarrow\infty}f\circ\varphi(t_{n}s_{k})$ and
     $lim_{k\rightarrow\infty}lim_{n\rightarrow\infty}f\circ\varphi(t_{n}s_{k})$
     exist, we have
   \begin{eqnarray*}
    lim_{n\rightarrow\infty}lim_{k\rightarrow\infty}
    f\circ\varphi(t_{n}s_{k})&=&lim_{n\rightarrow\infty}lim_{k\rightarrow\infty}
    f(\varphi(t_{n})\varphi(s_{k}))\\
    &=&lim_{k\rightarrow\infty}lim_{n\rightarrow\infty}
    f(\varphi(t_{n})\varphi(s_{k}))\\
   &=& lim_{k\rightarrow\infty}lim_{n\rightarrow\infty}
    f\circ\varphi(t_{n}s_{k})).
    \end{eqnarray*}

$f\circ{\widetilde{\varphi}}$ is a continuous extension of
$f\circ\varphi$ to $S^{Lmc}$ so
$f\circ\widetilde{\varphi}=\widehat{f\circ\varphi}$. But then
$\widehat{f\circ\varphi}(\p)\neq\widehat{f\circ\varphi}(\q)$ so that
$\p$ is not equivalent with $\q$, a contradiction.

To complete the proof, we show that $\eta$ is a homomorphism. Since
$\widetilde{\varphi}$ is a homomorphism also since $\eta\circ\gamma$
is a continuous extension of $\varphi$, we have
$\eta\circ\gamma=\widetilde{\varphi}$. Thus, given $\p$ and $\q$ in
$S^{Lmc}$, we have
\begin{eqnarray*}
\eta(\cap[\p]\dot{+}\cap[\q])&=&\eta(\cap[\p*\q])\\
  &=&\eta\circ\gamma(\p*\q)\\
  &=&\widetilde{\varphi}(\p*\q)\\
  &=&\widetilde{\varphi}(\p)\widehat{\varphi}(\q)\\
  &=&\eta\circ\gamma(\p)\eta\circ\gamma(\q)\\
  &=&\eta(\cap[\p])\eta(\cap[\q]).
  \end{eqnarray*}
  \end{proof}

\begin{thm}
Let $\Gamma$ be a set of pure $z-$filter. There exist an operator
$*$ on $\Gamma$ and a function $h:S\rightarrow\Gamma$ satisfying
conditions $(a)(i)$ and $(ii)$ of  Theorem 2.10 such that
$(\Gamma,*,h)$ is $S^{wap}$ if and only if $\Gamma$ satisfies
conditions $(a)$, $(c)$ and $(d)$ of Theorem 3.7.
\end{thm}
\begin{proof} (Sufficiency) By Theorem 3.7,
$\Gamma=\{\cap[\p]:\p\in S^{Lmc}\}$ so Theorem 3.12
applies.

(Necessity) By Theorem 2.10, $\Gamma$ is a quotient of
$S^{Lmc}$. Thus by Remark 2.9, the function $\gamma
:S^{Lmc}\rightarrow\Gamma$ defined by $\gamma(\p)\subseteq
\p$ is (well defined and) a quotient map. By condition $(a)(i)$, for
each $s\in S$, $s\in \cap h(s)$. That is, $h(s)\subseteq
\widehat{s}$. Then $h(s)=\gamma(\widehat{s})=e(s)$ so $h=e$.

Define an equivalence relation $R$ on $S$ by $pRq$ if and only if
$\gamma(\p)=\gamma(\p)$. It suffices to show that $\p R\q$ if and
only if $\p\equiv\q$. For then we get $\Gamma=\{\cap[\p]:\p\in
S^{Lmc}\}$ and Theorem 3.7 applies. To this end, let
$\p,\q\in S^{Lmc}$ and assume $\p R\q$. Suppose
$\p\equiv q$ is not true and pick $f\in wap(S)$ such that
$\widehat{f}(\p)\neq \widehat{f}(\q)$. By Theorem 2.2, there
exists $g\in \mathcal{C}(\Gamma)$ such that $g\circ e=f$. Then
$g\circ \gamma$ is continuous extension of $f$ to
$S^{Lmc}$. (For $s\in S,$ $(g\circ
\gamma)(\widehat{s})=g(e(s))=f(s)$.) Thus $g\circ
\gamma=\widehat{f}$. Since $\p R\q$, then
$\widehat{f}(\p)=g(\gamma (\p))=g(\gamma (\q))=\widehat{f}(\q),$
contradiction. Now assume that $\p\equiv\q$ and suppose that
$\gamma(\p)\neq\gamma(\q)$. Pick $g\in \mathcal{C}(\Gamma)$ such
that $g(\gamma(\p))\neq g(\gamma(\q))$. Define $f\in C(S)$ by
$f(s)=g(e(s))$. Then, since $f$ extends continuously to $\Gamma$,
$f\in wap(S)$ by Theorem 2.2. But, as above $g\circ \gamma=f$ so
$\widehat{f}(\p)\neq \widehat{f}(\q)$, a
contradiction.\end{proof}

We now proceed to apply our results to a determination of size of
$S^{wap}$.
\begin{defn}
Let $A\subseteq S$. We say $A$ is an unbounded subset of $S$ if
\begin{equation*}  \overline{A}\cap S^{*}\neq\emptyset,\end{equation*}
 where $S^*=S^{Lmc}-S$. Also a
sequence in $S$ is unbounded if range of sequence is unbounded.
\end{defn}
\begin{lem}
Let $A\in Z(Lmc(S))$. Then

a) Let there exist $\widetilde{p}$ and $\widetilde{q}$ in
$S^*$ such that $A\in \p+\q$. Then there exist
one-to-one sequences $\{t_n\}$ and $\{s_n\}$ in $S$ such that
$\{s_kt_n:k\leq n\}\subseteq A$.

 b) Let there exist
one-to-one unbounded sequences $\{t_n\}$ and $\{s_n\}$ in $S$ such
that $\{s_kt_n:k\leq n\}\subseteq A$. Then there exist $\p$ and
$\q$ in $S^*$ such that $\p*\q\in \overline{A}$.
\end{lem}
\begin{proof} $(a)$ Pick $\p$ and $\q$. Let
$B=\Omega_{\p}(A)=\{s\in S:\lambda_s^{-1}(A)\in \p\}$. Then $\q\in
\overline{B}$ and since $\q\notin S$, $B$ is unbounded. Let
$\{s_n\}$ be a one-to-one sequence in $B$. For each $n$, we have
$\bigcap_{k=1}^{n}\lambda_{s_k}^{-1}(A)\in\p$ and is hence
unbounded. Pick $t_n\in\bigcap_{k=1}^{n}\lambda_{s_k}^{-1}(A)$
such that $t_n\notin\{t_k:k<n\}$.

$(b)$ It is obvious. \end{proof}

\begin{lem}
Let $S$ be a commutative semitopological semigroup and let $\p\in S^*-cl_{S^{Lmc}}(S^**S^*)$. Then for each $\q\in
S^{Lmc}-\{\p\}$ there exists $f\in wap(S)$ such that
Range$(f)=[0,1]$, $\widehat{f}(\p)=1$ and $\widehat{f}(\q)=0$.
Consequently $[\p]=\{\p\}$.
\end{lem}
\begin{proof} Pick $A\in(\p)^{\circ}$ such that
$\overline{A}\cap cl_{S^{Lmc}}(S^**S^*)=\emptyset$. Pick $B\in(\p)^{\circ}$
such that $\q\notin \overline{B}$. Then there exists $f\in
Lmc(S)$ such that $\widehat{f}|_{A\cap B}=1$, $\widehat{f}((\overline{A})^c)\subseteq [0,1)$,
$\widehat{f}=0$ on $cl_{S^{Lmc}}(S^**S^*)\cup\{\widetilde{q}\}$ and $f(S)=[0,1]$. It suffices to show that
$f\in wap(S)$. Suppose, instead, we  have sequences $\{t_n\}$ and
$\{s_n\}$ such that $lim_{k\rightarrow \infty}lim_{n\rightarrow
\infty}f(s_kt_n)=1$ and $lim_{n\rightarrow
\infty}lim_{k\rightarrow \infty}f(s_kt_n)=0$. We may assume by
thinning that $\{t_n\}$ and $\{s_n\}$ are one-to-one and then,
\begin{equation*}  \{s_kt_n:k\leq
n\mbox{ for all sufficiently large }n,k\in \mathbb{N}\}\subseteq \overline{A}
\end{equation*}
 so if $\{s_n\}$ and $\{t_n\}$ be unbounded sequences, by Lemma 3.15(b),
$\overline{A}\cap (S^**S^*)\neq\emptyset$ and this is a contradiction.\\
If $\{s_n\}$ or $\{t_n\}$ be bounded in $S$ and limits are exist then limits are equal.
 \end{proof}

 \begin{thm}
 Let $S$ be a Hausdorff non-compact commutative semitopological semigroup and assume $A\in Z(Lmc(S))$ be an unbounded set such
that \\
$(i)$ card$(S^{Lmc})=$card$(\overline{A})$ and\\
$(ii)$ $\overline{A}\cap cl_{S^{Lmc}}(S^**S^*)=\emptyset$.\\
 Then card$(S^{wap})=$card$(S^{Lmc})$.
\end{thm}
\begin{proof}
 It is obvious that
card$(S^{wap})\leq$card$(S^{Lmc})$. By Lemma 3.16, for each $\p\in \overline{A}$, $[\p]=\{\p\}$.
Therefore $card(S^{wap})\geq card(\overline{A})=card(S^{Lmc}).$
\end{proof}
\begin{exam}
Let $S$ be a dense subsemigroup of $((1,+\infty),+)$. Then
card($S^{wap})=$ card$(S^{Lmc})$.
Since $cl_{S^{Lmc}}((1,2]\cap S)$ is infinite.
Let $\p,\q\in cl_{S^{Lmc}}((1,2]\cap S)-S$ then there
exist two nets $\{x_\alpha\}$ and $\{y_\beta\}$ in $(1,2]$ such that
$x_\alpha\rightarrow \p$ and $y_\beta\rightarrow \q$. This implies
that $\p*\q\notin cl_{S^{Lmc}}((1,\frac{3}{2}]\cap S)$. Therefore
\begin{equation*}
 S^*-cl_{S^{Lmc}}
(S^**S^*)\neq\emptyset.
\end{equation*}
It is obvious $card(cl_{S^{Lmc}}((1,2]\cap
S))=card(cl_{S^{Lmc}}[2,+\infty)\cap S)$ and so
\begin{equation*}  card(cl_{S^{Lmc}}((1,2]\cap S))=
card(S^{Lmc}).\end{equation*}
Now let $A=(1,2]\cap S$, then $A$ satisfy in assumptions of Theorem 3.17. Thus by Theorem 3.17,
card$(S^{wap})=$card$(S^{Lmc}).$
\end{exam}

 \begin{lem}
Let $S$ be a semitopological semigroup and let
 \begin{equation*}  \p\in S^*-cl_{S^{Lmc}}(S^**S^{Lmc}).
 \end{equation*}
 Then for each $\q\in
S^{Lmc}-\{\p\}$ there exists $f\in wap(S)$ such that
Range$(f)=[0,1]$ and $\widehat{f}(\p)=1$ and $\widehat{f}(\q)=0$.
Consequently $[\p]=\{\p\}$.
\end{lem}
\begin{proof}
 Pick $A\in(\p)^{\circ}$ such that
$\overline{A}\cap cl_{S^{Lmc}}(S^**S^{Lmc})=\emptyset$. Pick $B\in(\p)^{\circ}$
such that $\q\notin \overline{B}$. Then there exists $f\in
Lmc(S)$ such that $\widehat{f}|_{A\cap B}=1$, $\widehat{f}((\overline{A})^c)\subseteq [0,1)$,
$\widehat{f}=0$ on $cl_{S^{Lmc}}(S^**S^{Lmc})\cup\{\widetilde{q}\}$ and $f(S)=[0,1]$. It suffices to show that
$f\in wap(S)$. Suppose, instead, we  have sequences $\{t_n\}$ and
$\{s_n\}$ such that $lim_{k\rightarrow \infty}lim_{n\rightarrow
\infty}f(s_kt_n)=1$ and $lim_{n\rightarrow
\infty}lim_{k\rightarrow \infty}f(s_kt_n)=0$. We may assume by
thinning that $\{t_n\}$ and $\{s_n\}$ are one-to-one and that,
\begin{equation*}
 \{s_kt_n:k\leq
n\mbox{ for all sufficiently large }n,k\in \mathbb{N}\}\subseteq \overline{A},
\end{equation*}
 so if $\{s_n\}$ and $\{t_n\}$ be unbounded sequences, by Lemma 3.15(b),
$\overline{A}\cap (S^**S^*)\neq\emptyset$ and this is a contradiction.\\
Let $\{s_n\}$ and $\{t_n\}$ are bounded sequences in $S$ or  $\{s_n\}$ be a bounded sequence
and $\{t_n\}$ be an unbounded sequence in $S$,
 then \begin{equation*}  lim_{k\rightarrow \infty}lim_{n\rightarrow
\infty}f(s_kt_n)=lim_{n\rightarrow
\infty}lim_{k\rightarrow \infty}f(s_kt_n).
\end{equation*}
Now let $\{s_n\}$ be an unbounded sequence and $\{t_n\}$ be a bounded sequences in $S$, so
there exist two nets $\{s_\alpha\}\subseteq \{s_n:n\in \mathbb{N}\}$ and
 $\{t_\beta\}\subseteq \{t_n:n\in \mathbb{N}\}$ such that $s_\alpha\rightarrow u\in S^*$ and $t_\beta\rightarrow t\in S$.
 Therefore
 \begin{eqnarray*}
 lim_{k\rightarrow \infty}lim_{n\rightarrow
\infty}f(s_kt_n)&=&lim_{\alpha}lim_{\beta}f(s_\alpha t_\beta)\\
&=&lim_{\alpha}f(s_\alpha t)\\
&=&\widehat{f}(ut)\\
&=& 0.
\end{eqnarray*}
and
\begin{eqnarray*}
 lim_{n\rightarrow \infty}lim_{k\rightarrow
\infty}f(s_kt_n)&=&lim_{\beta}lim_{\alpha}f(s_\alpha t_\beta)\\
&=&lim_{\beta}\widehat{f}(ut_\beta)\\
&=& 0.
\end{eqnarray*}
Thus $f\in wap(S)$ and this complete proof.
\end{proof}
\begin{thm}
Let $S$ be a Hausdorff non-compact semitopological semigroup and assume $A\in Z(Lmc(S))$ be an unbounded set such
that\\
 $(i)$ card$(S^{Lmc})=$card$(\overline{A})$ and\\
$(ii)$ $\overline{A}\cap cl_{S^{Lmc}}(S^**S^{Lmc})=\emptyset$.\\
 Then
card$(S^{wap})=$card$(S^{Lmc})$.
\end{thm}
\begin{proof}
It is obvious that
card$(S^{wap})\leq$card$(S^{Lmc})$. By Lemma 3.19, for each $\p\in \overline{A}$, $[\p]=\{\p\}$.
Therefore $card(S^{wap})\geq card(\overline{A})=card(S^{Lmc}).$
\end{proof}
\begin{thm}
Let $S$ be a semitopological semigroup. Let $S^{wap}$ be the one point compactification and $card(S^*)>1$. Then
$S^**S^{Lmc}$ is dense in $S^*$.
\end{thm}
\begin{proof}
 Let $S^{wap}$ be the one point compactification so $[\p]=S^*$ for each $\p\in S^*$. Let
$\p\in S^*-cl_{S^{Lmc}}(S^**S^{Lmc})$ then for each $\q\in
S^{Lmc}-\{\p\}$ there exists $f\in wap(S)$ such that
Range$(f)=[0,1]$ and $\widehat{f}(\p)=1$ and $\widehat{f}(\q)=0$, by Lemma 3.19.
Consequently $[\p]=\{\p\}$. This is a contradiction. So
 $ S^*-cl_{S^{Lmc}}(S^**S^{Lmc})=\emptyset$.\end{proof}

\begin{exam}
 Let $G$ be the linear group $Sl(2,\mathbb{R})$. For this group
$G^{wap}$ = the one-point compactification,( see \cite{filaweak} and \cite{filali}), and so $card(G^{wap})=card(G)$
while $card(G^{Lmc}) =card(G^{\mathcal{LUC}}) = 2^{2^{\kappa (G)}}$ , where $G^{\mathcal{LUC}}$
is the maximal ideal space of the $C^*-$algebra of bounded left norm continuous function on $G$ and
$\kappa (G) =\omega$ (in this
case) is the compact covering number of G, i.e., the minimal number
of compact sets needed to cover $G$. Thus $[\p]=G^*$ for each $\p\in G^*=G^{\mathcal{LUC}}-G=G^{Lmc}-G$.
Hence $G^*G^{\mathcal{LUC}}=G^*G^{Lmc}$ is dense in $G^*$, by Theorem 3.21.
\end{exam}


\bibliographystyle{amsplain}

\end{document}